\documentclass[11pt]{amsart}

\usepackage{amsmath,amsthm,amssymb,amssymb, paralist, xspace, graphicx, url, amscd, euscript, mathrsfs,stmaryrd,epic,eepic,color}

\usepackage[all]{xy}
\usepackage{amsmath,amscd}
\SelectTips{cm}{}

\usepackage[colorlinks=true]{hyperref}

\usepackage{pstricks}

\numberwithin{equation}{subsection}

1

\numberwithin{subsection}{section}

\allowdisplaybreaks[1]

\newenvironment{enumeratei}
{\begin{enumerate}[\upshape (i)]}
{\end{enumerate}}

\newtheorem*{namedtheorem}{\theoremname}
\newcommand{\theoremname}{testing}

\theoremstyle{plain}

\newtheorem{thm}{Theorem}[section]
\newtheorem{proposition}[thm]{Proposition}
\newtheorem{proposition-definition}[thm]{Proposition-Definition}
\newtheorem{lemma-definition}[thm]{Lemma-Definition}
\newtheorem{corollary}[thm]{Corollary}
\newtheorem{lemma}[thm]{Lemma}

\theoremstyle{definition}
\newtheorem{definition}[thm]{Definition}
\newtheorem{notation}[thm]{Notation}

\newtheorem{assumption}[thm]{Assumption}

\theoremstyle{remark}

\numberwithin{thm}{section}

\newcommand\tC{\tilde{C}}

\newcommand\cA{\mathcal{A}}

\newcommand\cC{\mathcal{C}}

\newcommand\cM{\mathcal{M}}

\newcommand\cO{\mathcal{O}}

\newcommand\cR{\mathcal{R}}

\newcommand\cU{\mathcal{U}}

\newcommand\cX{\mathcal{X}}

\def\O{\mathcal{O}}

\def\P{\mathbb{P}}
\def\A{\mathbb{A}}

\def\C{\mathbb{C}}
\def\o{\mathfrak{o}}

\def\u{\underline}

\newcommand\uB{\underline{B}}
\newcommand\uC{\underline{C}}
\newcommand\uD{\underline{D}}
\newcommand\uE{\underline{E}}
\newcommand\uF{\underline{F}}
\newcommand\uf{\underline{f}}

\newcommand\uH{\underline{H}}

\newcommand\uL{\underline{L}}
\newcommand\uM{\underline{M}}

\newcommand\uS{\underline{S}}

\newcommand\uU{\underline{U}}
\newcommand\uV{\underline{V}}

\newcommand\uX{\underline{X}}
\newcommand\uY{\underline{Y}}
\newcommand\uZ{\underline{Z}}

\newcommand\CC{\mathbb{C}}

\newcommand\NN{\mathbb{N}}
\newcommand\PP{\mathbb{P}}

\newcommand\fM{\mathfrak{M}}

\newcommand\fro{\mathfrak{o}}

\newcommand\arr{\ifinner\to\else\longrightarrow\fi}

\def\displaytimes_#1{\mathrel{\mathop{\times}\limits_{#1}}}

\def\displayotimes_#1{\mathrel{\mathop{\bigotimes}\limits_{#1}}}

\newcommand\spec{\operatorname{Spec}}
\newcommand\Spec{\operatorname{Spec}}

\newdir{ >}{{}*!/-5pt/@{>}}

\newcommand\doublelong[2]{\mathbin{\xymatrix{{}\ar@<3pt>[r]^{#1}
\ar@<-3pt>[r]_{#2}&}}}

\newlength{\ignora}

\renewcommand{\setminus}{\smallsetminus}

\numberwithin{equation}{subsection}

\newcommand{\tf}{{\tilde{f}}} 

\newcommand\uev{\underline{ev}} 

\newcommand\oM{\overline{M}} 

\newcommand{\lM}{\mathcal{A}}  

\begin{document}


\title{Strong Approximation over Function Fields}

\author{Qile Chen}

\author{Yi Zhu}

\address[Chen]{Department of Mathematics\\
Boston College\\
Chestnut Hill,\\
MA 02467-3806\\
U.S.A.}
\email{qile.chen@bc.edu}

\address[Zhu]{Pure Mathematics\\Univeristy of Waterloo\\Waterloo, ON N2L3G1\\ Canada}
\email{yi.zhu@uwaterloo.ca}

\thanks{Chen is supported by NSF grant DMS-1403271 and DMS-1560830.}

\date{\today}

\subjclass[2010]{Primary 	14G05, 14G25, Secondary 14M10}
\keywords{strong approximation, affine varieties, $\A^1$-curve, stable log map, complete intersection}

\begin{abstract}


By studying $\A^1$-curves on varieties, we propose a geometric approach to strong approximation problem over function fields of complex curves. We prove that strong approximation holds for smooth, low degree affine complete intersections with the boundary smooth at infinity.


\end{abstract}

\maketitle

\tableofcontents

\section{Introduction}\label{sec:intro}

Given a variety $X$ over a number field, the existence of rational points (integral points) and their distributions (weak approximation, strong approximation) are extensively studied by number theorists. In general, these problems are very difficult and lacking of complete solutions. 

The analogue between number fields and function fields of Riemann surfaces suggests that one may start with geometric function field case first. A great progress has been made along this direction on the arithmetics of {projective} varieties with lots of rational curves during the past fifteen years. Let $F$ be a function field of a Riemann surface, and consider a projective variety $X$ defined over $F$. If the geometric fiber is rationally connected, then $X$ admits rational points \cite{GHS,dJS1}. When the geometric fiber is rationally simply connected, weak approximation holds for $X$, see \cite{dJS,Hassett-WA} for the definitions and results. Furthermore, it is expected that weak approximation holds for rationally connected varieties \cite{KMM,HT06}. As the first example, all the geometric conditions above hold for projective spaces.

While the results above focus on the projective case, number theorists study arithmetics of open varieties such as linear algebraic groups and affine hypersurfaces as well. This paper is an attempt to build a parallel theory of integral points on open varieties over $F$ from the logarithmic geometry point of view.

Since affine spaces satisfies strong approximation \cite[Theorem 6.13]{Rosen2002}, it is natural to search geometric substitute for affine spaces that strong approximation holds. As the non-proper generalization of rationally connected varieties, $\A^1$-connected varieties have been studied in our previous works \cite{CZ, A1, rankone,Z3}. On one hand side, this program is influenced by Iitaka's philosphy and the works of Keel-McKernan \cite{KM} by studying the birational geometry of pairs. On the other hand side, we replace the non-proper smooth varieties by their logarithmically smooth (equivalently toroidal) compactifications, and view $\A^1$-curves as a morphism of logarithmic schemes. The theory of stable log maps developed in \cite{Chen, AC, GS} provides a frame work for the study of $\A^1$-curves. 

In this paper, we introduce the notion {\em $\A^1$-simple-connectedness} using the stable log map compactification. This is a key to our approach for the strong approximation conjecture over $F$. Our proposal is parallel to the approach of Hassett \cite{Hassett-WA} and de Jong-Starr \cite{dJS} for weak approximation of rationally simply connected varieties.



{\begin{thm}\label{thm:SA}
Let $F$ be a function field of a smooth, irreducible complex algebraic curve. Let $X=(\uX,\uD)$ be a log smooth projective variety over $F$. Assume the following hold:
\begin{enumerate}
\item $\uX$ satisfies weak approximation over $F$.
\item There exists a curve class $\beta$ and a geometrically irreducible component $\uM$ of the moduli space $\cA_2(X,\beta)$ of two pointed $\A^1$-curves defined over $F$ such that
\begin{itemize}
\item a general point of $\uM$ parametrizes a smoothly embeded $\A^1$-curve.
\item The evaluation $$ev:\uM\to \uX\times \uX$$
is dominant with rationally connected geometric generic fiber.
\end{itemize}

\end{enumerate}
Then strong approximation holds for the interior $\uX\backslash \uD$ over $F$.
\end{thm}}

We refer to Section \ref{ss:notations} for the notations and terminologies of the above theorem. The formulation of strong approximation is defined in Section \ref{sec:strongapp}. If we call Condition (2) above \emph{$\A^1$-simple-connectedness} (with respect to the curve class $\beta$), the theorem above  states that strong approximation holds for $\A^1$-simply-connected $F$-varieties if weak approximation holds. Furthermore, $\A^1$-simple-connectedness is a geometric condition, and only depends on the interior.

Affine spaces are the first class of examples of $\A^1$-simply-connected varieties because any pair of points can be joined by a unique affine line. 

\begin{proposition}\label{prop:affine}
Affine spaces are $\A^1$-simply-connected. Thus strong approximation holds for affine spaces over $F$.\qed
\end{proposition}


By studying the geometry of $\A^1$-conics on complete intersection, we give a bound for low degree smooth complete intersection pairs to be $\A^1$-simply connected. 

\begin{thm}\label{thm:A1-conic-moduli}
	Let $X:=(\uX,\uD)$ be a smooth complete intersection pair of type $(d_1,\cdots,d_c;1)$ in $\P^n$. Assume that $\uX-\uD$ is not the affine space. Then the general fiber of the evaluation morphism defined in (\ref{not:evaluation})
	\begin{equation*}
	\u{ev}: \u\cA_{2}(X,2\alpha)\to \uX\times \uX
	\end{equation*}
	is a smooth complete intersection in $\P^n$ of type 
	\[
	(1,1,\cdots,d_1-1,d_1-1,d_1,\cdots,1,1,\cdots,d_c-1,d_c-1,d_c).
	\]
In particular, the general fiber is rationally connected if $\sum_{i=1}^c d_i^2\le n$.
\end{thm}

We refer to Section \ref{ss:notations} for the notations and terminologies of the above proposition. Combining this result with the works of \cite{Hassett-WA,dJS} on weak approximation of low degree complete intersection in projective spaces, we conclude that

\begin{thm}\label{thm:main}
Strong approximation holds for the interior of any {smooth complete intersection pair} of type $(d_1,\cdots, d_c;1)$ in $\P_F^n$ with $$\sum_{i=1}^c d_i^2\le n.$$
\end{thm}

\begin{corollary}\label{cor:main}
Strong approximation holds for the universal cover of the interior of any {smooth complete intersection pair} of type $(d_1,\cdots, d_c;k)$ in $\P_F^n$ with $$\sum_{i=1}^c d_i^2+k^2\le n+1.$$
\end{corollary}

It is known that there exists rational curves on a smooth projective rationally connected variety through any finite number of points. But the analogy for $\A^1$-connected varieties are wildly open. Our results on strong approximation provides an interesting class of  examples from this point of view.

\begin{corollary}\label{cor:geom}
Let $(\uX,\uD)$ be a smooth complete intersection pair of type $(d_1,\cdots, d_c;k)$ in $\P^n_\C$ with $\sum_{i=1}^c d_i^2+k^2\le n+1.$ Then there exists an $\A^1$-curve $(\uX,\uD)$ passing through any $m$-tuple of points on $\uX-\uD$.
\end{corollary}

In Section \ref{sec:strongapp}, we will state the geometric version of strong approximation, and proof Theorem \ref{thm:SA}. In Section \ref{sec:A1-line} and \ref{sec:A1-conic}, we analyze the moduli space of $\A^1$-lines and $\A^1$-conics, and conclude the proof of Theorem \ref{thm:main}, and Corollary \ref{cor:main} and \ref{cor:geom} in Section \ref{sec:proof-conclude}.

\subsection{Notations and terminologies}\label{ss:notations}

Capital letters such as $X$, $Y$, $Z$, and $C$, ect. are reserved for log schemes with the corresponding underlying schemes denoted by $\uX$, $\uY$, $\uZ$, and $\uC$. For any log scheme $X$, denote by $X^{\circ} \subset X$ the open locus with the trivial log structure.

An {\em $\A^1$-map} is a genus zero stable log map with precisely one marked point with a non-trivial contact order. An {\em $\A^1$-curve} is an $\A^1$-map with an irreducible source curve, whose image has non-trivial intersection with the open locus of the target with the trivial log structure. We call an $\A^1$-curve an {\em $\A^1$-line} or an {\em $\A^1$-conic} if the curve class of the $\A^1$-curve is the line or the conic curve class respectively.

For any log scheme $Z$, any curve class $\beta$ on $\uZ$, and an positive integer $e$, denote by $\cA_{m}(Z,e\beta)$ the moduli stack of $\A^1$-maps to $Z$ with curve class $e\beta$, and $m$ markings with the trivial contact order. Then $\cA_{m}(Z,e\beta)$ is a log stack with the canonical log structure. Denote by $\u\cA_{m}(Z,e\beta)$ the underlying stack obtained by removing the log structure of $\cA_{m}(Z,e\beta)$. We have the evaluation morphism induced by the $m$-markings with the trivial contact order
\begin{equation}\label{not:evaluation}
\u{ev}: \u\cA_{m}(Z,e\beta) \to \uX\times\cdots\times\uX
\end{equation}
where the right hand side is $m$-copies of $\uX$.

Let $\cR_{m}(\uZ,e\beta)$ be the moduli space of $m$-pointed, genus zero stable maps to $\uZ$ with curve class $e\beta$. 


\begin{notation}\label{not:smooth-pair}
Let $\uX$ be a smooth complete intersection in $\P^n$ ($n\ge 2$) of type $(d_1,\cdots,d_c)$ with $d_i\ge 2$. Let $\uD \subset \uX$ be a smooth hypersurface section of degree $k$. We call the pair $X := (\uX,\uD)$ a {\em smooth complete intersection pair of type $(d_1,\cdots,d_c;k)$}. We denote by $X$ the log scheme associated to the pair $(\uX,\uD)$. Denote by $d=d_1+\cdots+d_c+k.$ Denote by $\alpha$ the line class on $\uX$. 
\end{notation}

We refer to \cite{KKato} for the basics of logarithmic geometry, and \cite{Kato2000, Olsson2007} for the canonical log structures on curves. For the detailed development of stable log maps, the reader should consult \cite{Chen, AC, GS}.

\subsection*{Acknowledgments}
The authors would like to thank Xuanyu Pan for explaining his thesis work \cite{Pan}.

	
	
	

\section{Strong Approximation}\label{sec:strongapp}

\subsection{The arithmetic formulation}

We first recall the adelic formulation of strong (weak) approximation over function fields of curves, see \cite{Hassett-WA}. 

Let $B$ a smooth irreducible projective curve over $\C$ with function field $F=\C(B)$. For each place $v \in B$, denote by $F_{v}$ the completion of $F$ at $v$. Let $S$ be a nonempty finite set of places of $F$,  $\fro_{F,S}$ the ring of $S$-integers. Denote by $\A_{F,S} := \prod_{v \in B\setminus \{S\}}' F_v$ the ring of adeles over all places outside $S$, where the product is the restricted product, i.e. all but finite number of factors are in $\o_{v}$. The ring $\A_{F,S}$ has two natural topologies: the first one is the product topology, and the second one is the adelic topology, with a basis of open sets given by $\prod_{v \notin S} R_{v}$ where $R_v = \o_v$ for all but finitely many $v$. 

Let $U$ be a geometrically integral algebraic variety over $F$. Denote by $U(F)$ be the set of $F$-rational points, and $U(\A_{F,S})$ be the restricted product of $\prod'_{v\notin S}U(F_v)$. Thus, the adelic points $U(\A_{F,S})$ admits the product topology and adelic topology, i.e., locally inheriting from adelic affine spaces.

\begin{definition}\label{def:ss-arithmetic}
	We say that \emph{strong approximation} (respectively, \emph{weak approximation}) holds for $U$ away from $S$ if the inclusion
	\[
	U(F) \to U(\A_{F,S})
	\]
	is dense in the adelic topology (respectively, product topology). To be more precise, this is equivalent to for any finite set $T$ of places containing $S$, any integral model $\cU$ over $\fro_{F,S}$ of $U$, and any open set $W_v\subset U(F_v)$ under the adelic topology for each place $v\in T\backslash S$, the image of $U(F)$ via the diagonal map in
	\begin{equation}\label{eq:SA}
	\prod_{v\in T\backslash S} W_v\times \prod_{v\notin T} \cU(\o_v)(\text{respectively, }\prod_{v\in T\backslash S} W_v\times \prod_{v\notin T} U(F_v))
	\end{equation}
	is not empty. We say that \emph{strong approximation holds for} $U$ if strong approximation holds for $U$ away from any nonempty $S$.
	
	
\end{definition}

The above definition does not depend on the choice of model. Strong approximation implies weak approximation. The converse also holds when $U$ is proper over $F$.




\subsection{The geometric formulation} 
The geometric setting of weak approximation has been formulated and studied in \cite{HT06}. We next translate Definition \ref{def:ss-arithmetic} into the geometric setting. To apply logarithmic geometry, we would like to replace the open variety $U$ by a proper log smooth variety $X$ with the log trivial part $U$. 

\begin{definition}\label{def:int-model}
	Let $X$ be a smooth, proper, and log smooth variety over $F$. Denote by $U = X^{\circ}$ its log trivial open subset. A {\em proper model} of $X$ is a family of log schemes:
	\[
	\pi: \cX \to B
	\]
	such that 
	\begin{enumerate}
		\item $B$ is a smooth projective curve with the trivial log structure;
		\item $\u\pi:\u\cX\to \uB$ is proper flat over $\uB$;
		\item the generic fiber of $\pi$ is $X$.
	\end{enumerate}
	We say such model is \emph{regular} if $\u\cX$ is a smooth variety. This can always be achieved via resolution of singularities.
\end{definition}

\begin{proposition}\label{prop:red-SA}
	Let $U$ be the log trivial open subset of a proper, smooth, log smooth variety $X$ defined over $F$. Then strong approximation holds for $U$ away from $S$ is equivalent to the following statement:\\
	Given any proper regular model of $X$ as in Definition \ref{def:int-model}, any finite set of places $T=S\cup\{b_1,\cdots,b_k\}$ such that $\pi$ is smooth and log smooth over $\o_{F,T}$, any smooth points $x_i$ in $\underline{\cX}_{b_i}^{sm}$ for $i=1,\cdots,k$ can be realized by a section of $\pi$ which is integral (i.e., away from the boundary) over $\o_{F,T}$.

\end{proposition} 

\proof With the same notations as in (\ref{def:ss-arithmetic}), since $\cU(\o_v)$ is open in $X(F_v)$ for any $v\notin T$, we may enlarge the set $T$ such that the integral model $\cU$ can be embedded into a regular proper model $\cX$ of $X$ over $\o_{F,T}$ and $\cX$ is smooth and log smooth over $\o_{F,T}$. The rest is done by the iterated blow-ups of the jet datum \cite[2.3]{HT06}, \cite[1.5]{Hassett-WA}. \qed

\subsection{Proof of Theorem \ref{thm:SA}}\ \\

{\bf Step 1}  
To prove the theorem, it suffices to verify the statement in Proposition \ref{prop:red-SA}. Let $\uV\subset \uX\times \uX$ denote the open subset over which $\u{ev}$ has geometrically irreducible, and rationally connected fibers whose general points parametrizing $\A^1$-curves. 

\bigskip

{\bf Step 2}    
By assumption, we know that $X$ is $\A^1$-connected. In particular, $\uX$ is rationally connected. By \cite{GHS,KMM}, the rational points of $\uX$ over $F$ are dense. After enlarging $T$, there exists a rational section $$s:\uB\to\u\cX$$ such that \begin{itemize}
	\item $s$ is integral over $\o_{F,T}$;
	\item the associated rational point, still denoted by $s$, lies in $pr_1(\uV)$.
\end{itemize} 
\bigskip

{\bf Step 3}   
Since weak approximation holds over $F$, we may choose a general section $$t:\uB\to \u\cX$$ such that 
\begin{itemize}
	\item $t(b_i)=x_i$ for $i=1,\cdots, k$;
	\item the associated rational point, still denoted by $t$, lies in $pr_2(\uV)$.
\end{itemize}  
\bigskip

{\bf Step 4}  
The fiber $\u{ev}^{-1}(s,t)$ is a geometrically irreducible rationally connected variety defined over $F$ whose general points parametrize $\A^1$-curves. By \cite{GHS,KMM}, there exists a rational point on $\u{ev}^{-1}(s,t)$ parametrizing a smooth embedded $\A^1$-curve. This rational point gives a generic $\A^1$-ruled surface in $\u\cX$, denoted by $\uH\to \uB$. By construction, the surface $\uH$ contains:
\begin{itemize}
	\item the section $s$ integral over $\o_{F,T}$, and
	\item the section $t$. In particular, $H_{b_i}$ is smooth at $x_i$ for all $i$.
\end{itemize} 

Since strong approximation holds for $\A^1_F$ away from $S$ \cite[Theorem 6.13]{Rosen2002}, we can find a section $\sigma:\uB\to \u{H}\to \u\cX$ with $\sigma(b_i)=x_i$ for all $i$ and integral over $\o_{F,T}$.\qed

\section{$\A^1$-lines through a general point}\label{sec:A1-line}

\subsection{A deformation result}

\begin{proposition}\label{prop:generic-freeness}
	Let $X$ be a projective log smooth variety. For any curve class $\beta \in H_2(\uX)$ and a subscheme $B \in X^{\circ}$ with $B$ either a closed point or the empty set, there are finitely many sub-varieties $\{\uY_i\}$ of $X^{\circ}$ such that if $f: (\PP^1, \infty) \to X$ is an $\A^1$-curve with curve class $\beta $ through $B$, and $f(\PP^1\setminus\{\infty\}) \notin \uY_i$, then $f$ is free. In particular, an $\A^1$-curve through $B$ and a general point of $X^{\circ}$  with curve class $\beta$  is free.
\end{proposition}
\begin{proof}
	Denote by 
	\[\lM^{\circ}_{B}(X,\beta)=
	\begin{cases}
	\lM^{\circ}_{0}(X,\beta), \mbox{  if $B = \emptyset$, or}\\
	\uev^{-1}(B), \mbox{  if $B$ is a point,}
	\end{cases}
	\]
	where $\uev: \lM^{\circ}_{1}(X,\beta) \to \uX$ is the evaluation morphism induced by the marking with the trivial contact order.
	
	Let $Z_i$ be the irreducible component of $\lM^{\circ}_{B}(X,\beta)$ with the universal morphism $f^{\circ}_i: C_{i}^{\circ}:= C_i\setminus\{\infty\} \to X$. Let
	\[
	\uY = 
	\begin{cases}
	\overline{f^{\circ}_i(C_{i}^{\circ})}, \mbox{  if $f^{\circ}_i$ is not dominant, and}\\
	X^{\circ} \setminus U_i, \mbox{  if $f^{\circ}_i$ is dominant,}
	\end{cases}
	\]
	where $\uU_i \subset X^{\circ}$ is an open and dense subset such that $f^{\circ}$ is smooth over $\uU_i$, and all closures are taken in $X^{\circ}$. 
	
	Consider an $\A^1$-curve $f:  (\PP^1, \infty) \to X$ of curve class $\beta$ with $f(\PP^1\setminus\{\infty\}) \notin \uY_i$ for any $i$. Let $Z_j$ be the component containing $f$. By construction, the universal morphism $f^{\circ}_j$ is dominant, and $f$ intersects $\uU_j$. Same argument as in \cite[Chapter II 3.10]{Kollar} implies that $f$ is free. 
\end{proof}

\begin{corollary}
	Notations and assumptions as in Proposition \ref{prop:generic-freeness}, any $\A^1$-curve passing through $B$ and a very general point of $X^{\circ}$ is free.
\end{corollary}
\begin{proof}
	This follows from Proposition \ref{prop:generic-freeness} by taking into account all choices of curve classes.
\end{proof}

\subsection{$\A^1$-lines on smooth complete intersection pairs}

Consider the smooth complete intersection pair $X = (\uX, \uD)$ as in Notation \ref{not:smooth-pair}. In this subsection, we study the evaluation morphism 

$$\u{ev}:\u\cA_1(X,\alpha)\to \uX.$$

\begin{proposition}\label{prop:expected-dim}
\begin{enumerate}
	\item A general fiber of $\u{ev}$ is smooth and projective.
	\item Every nonempty connected component of a general fiber is of expected dimension $n-d$.
\end{enumerate}
\end{proposition}
\proof 
The first statement follows from Proposition \ref{prop:generic-freeness}. Since every $\A^1$-map with line class in a general fiber is free, the dimension is calculated by the Euler characteristic of the pullback of the log tangent bundle.
$$c_1(TX).\alpha+\dim \uX+2-3-\dim \uX=n+1-d-1=n-d.$$
\qed

Next we would like to describe the general fiber of $\u{ev}$ explicitly in equations. Fix a general point {$x\in X^{\circ}$}. Let ${\uL_{x}}$ be the fiber over $x$ of the evaluation morphism:
$$\u{ev}:\u\cA_1(X,\alpha)\to \uX.$$

We consider the restriction of the boundary evaluation morphism on $\uL_{x}$:
$$b':\uL_{x}\to \uD.$$

\begin{proposition}\label{prop:Fsch}
 If $d=d_1+\cdots+d_c+k\le n$, the morphism $b'$ is a closed immersion, and the image of $\uL_{x}$ is an irreducible, smooth complete intersection in $\P^n$ of type 
	$$(1,\cdots,d_1,\cdots,1,\cdots,d_c,1,\cdots, k).$$
\end{proposition}

Let $T=\Spec R$ be any affine scheme. For any scheme $\uZ$, we denote by $\uZ_T := \uZ\times T$. 

The scheme $L_x$ is determined by its $T$-points:
$$L_{x}(T)=\{ \A^1\text{-lines in }\uX_T \text{ through } x_T  \}.$$

Assume for simplicity $x=[1:0:\cdots:0]\in X^{\circ}$. Consider a $T$-point $q = [x_0:x_1:\cdots:x_n] \in \uD(T)$. A $T$-line $\ell$ joining $x_T$ and $q$ can be expressed as 
$$[t+x_0:x_1:\cdots:x_n], $$
{where $t$ is the parameter of the line and $x_i\in R$ for each $i$.}

Let $F_i$ be the defining equation of $\uX$ for $i=1,\cdots,c$ with $\deg F_i = d_i$. Restricting them on the line equation of $l$, we have
\begin{equation}\label{equ:line-on-uX}
F_i(t+x_0,x_1,\cdots,x_n) = P_{i0}\cdot t^{d_i} + P_{i1}\cdot t^{d_i - 1} + \cdots + P_{id_i}
\end{equation}
where $P_{ij}\in R[x_0,\cdots,x_n]$ is a homogeneous polynomial of degree $j$. The condition $y\in \uY$ implies that $P_{{ i 0}} = 0$. The condition $\ell \subset \uX_T$ is equivalent to the vanishing of $P_{i1}, \cdots, P_{id_i}$ for each $i$, which gives a complete intersection of type 
\[
(1, 2, \cdots, d_1, \cdots, 1, 2, \cdots d_c).
\]

Similarly, let $G$ be the defining equation of $\uD$. Restricting them on the line equation of $\ell$, we have:
\begin{equation}\label{equ:line-on-uD}
G(t+x_0,x_1,\cdots,x_n) = Q_{0}\cdot t^{d_i} + Q_{1}\cdot t^{d_i -1 } + \cdots + Q_{k},
\end{equation}
where $Q_{j}\in R[x_0,\cdots,x_n]$ is a homogeneous polynomial of degree $j$. The point $x_T$ lying outside $\uD_T$ implies that $Q_0\neq 0$. Note that $Q_k$ is indeed $G$. The condition being an $\A^1$-line is equivalent to the vanishing of the polynomials $Q_1,\cdots, Q_k$, i.e., a complete intersection of type $$(1,\cdots,k).$$

Now we define a complete interesection $\uZ$ in $\P^n_T$ defined by the equations:$$P_{11},\cdots, P_{1d_1},\cdots,P_{c1},\cdots,P_{cd_c},Q_1,\cdots,Q_k.$$ 
Since $Q_k=G$, $\uZ$ is automatically a closed subscheme of $\uD$.

To summarize, we proved the following.
\begin{lemma}
	The image of the morphism $$b':\uL_x\to \uD$$ lies in $\uZ$.
\end{lemma}

\proof[Proof of Proposition \ref{prop:Fsch}] It suffices to prove that $\uL_x$ is isomorphic to $\uZ$ under $b'$, i.e., every $T$-point of $\uZ$ is the image of a unique $T$-point of $\uL_x$ under $b'$. This follows from the fact that any $T$-point of $\uZ$ gives a $T$-family of lines via the projection:
$$pr:\P^n_T-\{x_T\}\to \P_T^{n-1},$$
where the target is the Hilbert scheme of lines through $y_T$. Furthermore, such family of lines meet the boundary exactly once, hence is a family of $\A^1$-lines.
\qed

\begin{corollary}
	A general fiber of $\u{ev}$ is a nonempty, irreducible, and smooth complete intersection if $X$ is log Fano, or equivalently, $d\le n$.
\end{corollary}

\proof This follows from Proposition \ref{prop:expected-dim} and \ref{prop:Fsch}.\qed









\section{Moduli of $\A^1$-conics through general two points}\label{sec:A1-conic}
For the rest of this section, we work with the following assumption. 
\begin{assumption}\label{not:deg1pair}
	Let $X=(\uX,\uD)$ be a smooth complete intersection pair in $\P^n_\C$ of type $(d_1,\cdots,d_c;1)$ with $d_i\ge 2$ for each $i$.
\end{assumption}

The goal of this section is to study general fibers of the two-pointed evaluation morphism 
\begin{equation}\label{equ:2-ev}
\u{ev}: \u\cA_{2}(X,2\alpha)\to \uX\times \uX
\end{equation}
given by the two marked points with the trivial contact order. The proof of Theorem \ref{thm:A1-conic-moduli} will be concluded at the end of this section.

For later use, denote by $\uF_{(p,q)}$ the fiber of (\ref{equ:2-ev}), and $F_{(p,q)}$ the corresponding log scheme with the minimal log structure pulled back from $\cA_{2}(X, 2\alpha)$. When there is no confusion of the pair of points $(p,q)$, we will simply write $\uF$ and $F$, and omit the subscripts.

\subsection{Smoothness of the moduli}

We first observe the following:
\begin{lemma}\label{lem:no-bad-line}
The line through a general pair of points $p,q$ in $\uX$ is not contained in $\uX$.\qed
\end{lemma}
	


\begin{lemma}\label{lem:reducible-conic}
For a general pair of points $(p,q) \in X^{\circ}\times X^{\circ}$, the boundary divisor $\u\Delta\subset \uF$ parameterizes maps with the following properties:

\begin{enumerate}
\item The underlying curve $\uC$ consists of three irreducible genus zero components $\uZ_{i}$ for $i=0,1,2$, with precisely two nodes $z_i$ joining $\uZ_0$ and $\uZ_i$ for $i=1,2$.

\item Each component $\uZ_i$ contains a marking $\sigma_i$ with the trivial contact order for $i = 1,2$.

\item $\uZ_0$ has three special points given by the contact marking and two nodes.

\item $f$ contracts the component $\uZ_0$ to a point $x_f := f(\uZ_0) \in D$.

\item The restriction $\uf|_{\uZ_i}$ is an embedding of two free $\A^1$-lines for $i=1,2$.

\item The characteristic sheaf $\oM_{F}|_{\u\Delta}$ is a locally constant sheaf with fiber $\NN$. 

\item Let $C^{\sharp} \to S^{\sharp}$ be the canonical log structure on $\uC \to \uS$, see  \cite{Kato2000, Olsson2007}. Then the canonical morphism $\oM_{S^{\sharp}} \to \oM_{S}$ is fiber-wise given by $\NN^2 \to \NN,  (a, b) \mapsto a+b$.
\end{enumerate}

\end{lemma}

\begin{proof}
By Lemma \ref{lem:no-bad-line}, the boundary $\u\Delta$ parameterizes stable log maps with property (1) -- (4).  Statement (5) of $\uf|_{\uZ_i}$ follows from Proposition \ref{prop:generic-freeness} and the general choice of $p,q$. Property (6) and (7) follow from the definition of minimality, see \cite[Construction 3.3.3]{Chen}, \cite[Section 4]{AC}, and \cite[Construction 1.16]{GS}.
\end{proof}

\begin{lemma}\label{lem:fiber-smoothness}
For a general pair of points $p,q$, the fiber $F$ is a log smooth scheme with the smooth boundary divisor $\u\Delta$ .
\end{lemma}

\begin{proof}
Let $U\subset \cA_{2}(X,2\alpha)$ be the open sub-stack parameterizing free $\A^1$-maps with the curve class $2 \alpha$.  Lemma \ref{lem:reducible-conic}(7) implies that $F \subset U$. Furthermore, the morphism $\u{ev}:\uU\to \uX\times \uX$ is smooth along the boundary divisor parametrizing reducible conics by Lemma \ref{lem:reducible-conic}(6). By generic smoothness, we conclude that $\uF$ is smooth with a smooth boundary divisor $\u\Delta$. In particular, the pair $F=(\uF,\u\Delta)$ is log smooth.
\end{proof}



\subsection{A lifting property in the transversal case}\label{ss:lifting}

We pause here to study the lifting of a special type of usual stable maps to stable log maps. Here is a slightly general result that fits our need.

\begin{proposition}\label{prop:lift}
Let $Z$ be a log smooth variety over $\CC$ with a smooth boundary divisor $D$. Consider a family of genus zero usual stable maps $\uf: \uC \to \uZ$ with two markings $\sigma_1$ and $\sigma_2$ over an arbitrary base scheme $\uS$ such that 
\begin{enumerate}
 \item The family $\uC \to \uS$ is obtained by gluing two families of smooth rational curves $\uC_1 \to \uS$ and $\uC_2 \to \uS$ along the markings $\infty_1 \subset \uC_1$ and $\infty_2 \subset \uC_2$.

 \item Each $\uC_{i} \to \uS$ has two markings $\sigma_i$ and $\infty_{i}$ for $ i = 1,2$.

 \item The restriction $\uf|_{\uC_i}$ is a family of $\A^1$-curves over $\uS$ intersecting $D$ transversally along $\infty_{i}$ for $i = 1, 2$.
\end{enumerate}
Then there exists up to a unique isomorphism, a unique family of genus zero minimal stable log maps $\tf: \tC/S \to Z$ such that
\begin{enumeratei}

 \item The underlying scheme of $S$ is $\uS$.

 \item The family of stable log maps has one contact marking $\infty$, and two other markings $\sigma_1, \sigma_2$ with the trivial contact order.

 \item The family of usual stable maps obtained by removing log structures on $\tf$, forgetting the contact marking $\infty$, and then stabilizing, is $\uf$.
\end{enumeratei}
\end{proposition}

We divide the proof into the following two lemmas. We first prove the local existence.

\begin{lemma}\label{lem:lift-existence}
Notations as in Proposition \ref{prop:lift}, the existence in Proposition \ref{prop:lift} holds locally over $S$.
\end{lemma}
\begin{proof}
Our construction here is similar to the case of \cite[Proposition 2.2]{rankone} but for a family of maps. We take a family of smooth rational curves $\uC_0 \to \uS$ with three markings $\infty$, $\infty'_1$, and $\infty'_2$. Such a family is necessarily trivial, and we thus have $\uC_0 \cong \P^1 \times \uS$. We have a family of nodal rational curves 
\[
\uC' \to \uS
\]
obtained by gluing $\uC_i$ with $\uC_0$ via the identification of the markings 
\[\infty_i \cong \infty'_i . \ \ \ \mbox{for } i = 1,2.\]
Now the underlying stable map $\uf$ over $\uS$ lifts uniquely to the underlying stable map 
\begin{equation}\label{equ:underlying-lift}
\u\tf: \uC' \to \uZ
\end{equation} 
over $\uS$ by contracting the component $\uC_0$. 

Consider the projectivized normal bundle $\PP:= \PP_{D}(N_{D/\uZ}\oplus \cO_{D})$ with two boundary divisors $D_{-}$ and $D_{+}$ corresponding to the normal bundles $N^{\vee}_{D/\uZ}$ and $N_{D/\uZ}$ respectively.  Here $N_{D/\uZ}$ is the normal bundle of $D$ in $\uZ$. Consider the expansion $\uZ[1]$ obtained by gluing $\uZ$ and $\PP$ via the identification $D \cong D_{-}$. We next want to lift $\u\tf$ to a stable map $\u\tf': \uC' \to \uZ[1]$ such that
\begin{enumerate}
 \item $\u\tf'|_{\uC_i} = \u\tf_{\uC_i}$ for $i = 1,2$.
 \item the composition $\uC_0 \to \PP \to D$ is compatible with $\u\tf|_{\uC_0}$.
 \item $\u\tf'|_{\uC_0}: \uC_0 \to \PP$ is a family of relative stable map tagent to $D_{+}$ only along $\infty$ with multiplicity $2$, and intersecting $D_{-}$ transversally only along $\infty_{1}$ and $\infty_2$. 
\end{enumerate}

Replace $S$ by an \'etale cover, we may assume that the pull-back $\PP_{\uS} = \uS \times_{D} \PP$ along $\infty_i \to D$ is a trivial family of rational curves over $\uS$. Note also that the morphism $\u\tf'|_{\uC_0}$ needed factors through $\PP_{\uS}$ with the corresponding tangency along $D_{-,\uS}:= \uS\times_{D}D_{-}$ and $D_{+,\uS}:= \uS\times_{D}D_{+}$. Since $\uC_0 \cong \P^1 \times \uS \to \uS$ is also a trivial family, to construct $\u\tf'_{\uC_{0}}$, it suffices to select a meromorphic section on $\P^1$ with two simple zeros along $\infty'_{1}$ and $\infty'_{2}$, and a pole of order $2$ along $\infty$. Such a meromorphic section clearly exists. This yields the usual stable map $\u\tf'$ with the desired properties.

Finally, by \cite{Kim}, see also the construction of \cite[Proposition 2.2]{rankone}, since the usual stable map $\u\tf'$ intersects the boundary transversally along $\infty_i$, it lifts to a unique log stable map $\tf': C' \to Z[1]$ over a log scheme $S$ in the sense of \cite{Kim}. Here $Z[1]$ is the log scheme with the underlying structure $\uZ[1]$, and the canonical log structure as in \cite{Olsson2003}, and $S$ has underlying structure $\uS$. Since there is a natural projection of log schemes $Z[1] \to Z$, composition this projection with $\tf'$ we obtain the stable log map $\tf$ as needed. 
\end{proof}

\begin{lemma}\label{lem:lift-uniqueness}
The uniqueness in Proposition \ref{prop:lift} holds.
\end{lemma}
\begin{proof}
It suffices to show the uniqueness locally. Shrinking $S$, we may again assume base scheme $\uS = \spec R$ is affine. It sufficies to verify that the lift constructed in Lemma \ref{lem:lift-existence} is unique.

Assume that we have two different liftings $\tf_1: \tC'_1/S_1 \to Z$ and $\tf_2: \tC'_2/S_2 \to Z$. We first notice that except the freeness in (5), all other statements in Lemma \ref{lem:reducible-conic} applies to both $\tf_1$ and $\tf_2$. In particular, the two liftings $\tf_1$ and $\tf_2$ have the same underlying stable map (\ref{equ:underlying-lift}) over $\uS$ constructed in the proof of Lemma \ref{lem:lift-existence}. 

We first compare the two stable log maps over the contracted component $\uC_0$. Shrinking $\uS$, we may assume that $\u\tf^*\cM_{Z}|_{\uC_0}$ is generated by a global section $\delta$. By Lemma \ref{lem:reducible-conic}(6) and further shrinking $\uS$, we may assume that $\cM_{S_i}$ is generated by a global section $e_i$ for $i=1,2$. By choosing the generators appropriately, we may assume over $\uC_0$ we have 
\begin{equation}\label{equ:over-contracted-curve}
\tf_i^{\flat}(\delta) = e_i + \log \sigma, \ \ \ \mbox{for $i=1,2$}
\end{equation}
where $\sigma$ is a meromorphic function on $\uC_0$ with only first order poles along $\infty'_1$ and $\infty'_2$, and second order zero along $\infty$.

We now focus on the node $\infty_i$ of $\u\tC'$ for $i=1,2$. Let $\cM_{S^{\sharp}}$ and $\cM_{\tC^{\sharp}}$ be the canonical log structure on $\uS$ and $\u\tC$ associated to the family $\u\tC \to \uS$. Shrinking $S$ again, we assume $\cM_{S^{\sharp}}$ is generated by global sections $a_1$ and $a_2$ corresponding to smoothing nodes $\infty_1$ and $\infty_2$ respectively. By Lemma \ref{lem:reducible-conic}(7), the log curves $\tC_1 \to S_1$ and $\tC_2 \to S_2$ has log structure near node $\infty_i$ defined by pulling the standard log structure along the correspondences respectively:
\begin{equation}\label{equ:curve-log-1}
a_i = e_1 + \log u_i
\end{equation}
and
\begin{equation}\label{equ:curve-log-2}
a_i = e_2 + \log v_i
\end{equation}
where $u_i, v_i \in R^*$. Here we identify $a_i$ with its image in $M_{S_i}$. 

Since $\u\tf|_{\uC_i}$ is an embedding of a family of $\A^1$-curves, the two sections $\tf_1^{\flat}(\delta)$ and $\tf_2^{\flat}(\delta)$ are identified with the image of $\u\tf^*(\exp(\delta))$, where $\exp(\delta) \in \cO_{\u\tC'}$ is the image of $\delta$. This in particular means that we have a canonical identification
\begin{equation}\label{equ:along-node}
\tf_1^{\flat}(\delta) = \tf_2^{\flat}(\delta)
\end{equation}
along the node $\infty_i$ for $i=1,2$. A calculation combining (\ref{equ:along-node}) with (\ref{equ:over-contracted-curve}), (\ref{equ:curve-log-1}), and (\ref{equ:curve-log-2}) implies that 
\[
u_i = v_i.
\]
We thus conclude that the morphism of log structures $\cM_{S_1} \to \cM_{S_2}$ induced by the correspondence $e_1 \mapsto e_2$ induces an isomorphism of the two log curves $\tC'_1 \to S_1$ and $\tC'_2 \to S_2$. In view of (\ref{equ:over-contracted-curve}), this further induces an isomorphism of the two stable log maps $\tf_1 \cong \tf_2$. Such isomorphism is canonical from the discussion above.
\end{proof}

To prove Proposition \ref{prop:lift}, we may first construct the log lifts locally using Lemma \ref{lem:lift-existence}, then glue the local construction together using Lemma \ref{lem:lift-uniqueness}. This proves the existence of lifting. The uniqueness follows from Lemma \ref{lem:lift-uniqueness}. \qed

\subsection{Forgetful morphism to moduli of usual stable maps}

Now consider the moduli space of usual stable maps with two markings $\u\cR_{2}(\uX,2\alpha)$. Consider the $2$-evaluation morphism
\begin{equation}\label{equ:2-ev-usual}
\uev: \u\cR_{2}(\uX,2\alpha) \to \uX\times \uX
\end{equation}
induced by the two markings. Given a pair of points $(p,q) \in \uX \times \uX$, denote by $\uF'_{(p,q)}$ the fiber of (\ref{equ:2-ev-usual}) over $(p,q)$. When there is no danger of confusion, we will write $\uF'$ instead of $\uF'_{(p,q)}$. Denote by $\u\Delta' \subset \uF'$ the locus parameterizing maps with reducible domain curves. Recall that

\begin{lemma}\cite[Lemma 5.1]{dJS}\label{lem:Kont-log-smooth}
For a general choice of $(p,q)$, the scheme $\uF'$ is smooth with a smooth divisor $\u\Delta'$.
\end{lemma}

We then consider the forgetful morphism 
\[
\u\Phi: \u\cA_{2}(X,2\alpha) \to \u\cR_{2}(\uX,2\alpha)
\]
obtained by sending a stable log map to its underlying stable map, forgetting the marking $\infty$ with non-trivial contact order, then stabilizing. This induces a forgetful morphism of the fibers
\begin{equation}\label{equ:fiber-forget}
\u\phi: \uF \to \uF'.
\end{equation}

\begin{proposition}\label{prop:embedding}
Fixing a general choice of $(p,q)$, the forgetful morphism $\u\phi$ is an embedding of a closed sub-scheme. Furthermore, it induces an embedding of closed sub-scheme $\u\Delta \to \u\Delta'$.
\end{proposition}
\begin{proof}
Note that if a usual stable map intersects the boundary of $X$ at a single smooth point of the source curve, then it lifts to an $\A^1$-map in a unique way. Thus, the morphism $\uF \setminus \u\Delta \to \uF'\setminus\u\Delta'$ is an embedding. It remains to consider around the locus of stable log maps with reducible domain curves. 

By Assumption (\ref{not:deg1pair}) and Proposition \ref{prop:lift}, the forgetful morphism $\u\phi$ is injective on the level of closed points. Since by our assumptions both $\uF$ and $\uF'$ are smooth, it remains to verify the injectivity of tangent vectors.  

We fixed a minimal stable log map $f: C/S \to X$ over a geometric point $\uS = \spec \CC \in \uF$.  It suffices to consider the case that $\uC$ is reducible, and denote by $f': C'/S' \to \uX$ the image of $f$ in $F'$.

Consider the morphism of fibers $\mbox{d}\phi_{[f]}: T_{\uF,[f]} \to T_{\uF',[f']}$. Recall that for any smooth variety $\uY$, the tangent bundle $T_{\uY}$ can be identified with $Hom(\spec \CC[\epsilon]/(\epsilon^2), \uY)$. Now that injectivity of tangent vectors follows from applying the uniqueness of  Proposition \ref{prop:lift} to families over $\uS[\epsilon] := \spec\CC[\epsilon]/(\epsilon^2)$.
\end{proof}

\subsection{Pull-back of the boundary divisor}

Denote by $F'$ the log scheme with underlying structure $\uF'$, and log structure given by the canonical one associated to the underlying curves, see  \cite{Kato2000, Olsson2007}. We note thate

\begin{lemma}\label{lem:log-conic-space}
Fix a general choice of $(p,q)$. The log smooth scheme $F'$ has its log structure given by the boundary divisor $\u\Delta'$.
\end{lemma}
\begin{proof}
Since the locus of $\uF'$ with reducible domain curves form the smooth divisor $\u\Delta'$, to show that the log structure of $F'$ is same as the log structure given by the smooth divisor $\u\Delta'$, it suffices to verify $F'$ is log smooth. Thus, it suffices to verify $F' \to \fM_{0,2}$ is log smooth, where $\fM_{0,2}$ is the Artin stack of genus zero pre-stable curves with two markings equipped with the canonical log structure of curves. Since the morphism $F' \to \fM_{0,2}$ is strict, the log smoothness is equivalent to the smoothness of the underlying maps $\uF' \to \u\fM_{0,2}$. This follows from that $\uF'$ parameterizes free usual stable maps.
\end{proof}

\begin{proposition}\label{prop:boundary-pullback}
Fix a general choice of $(p,q)$. There is a canonical morphism of log schemes
\begin{equation}\label{equ:fiber-log-forget}
\phi: F \to F'
\end{equation}
compatible with $\u\phi$ in (\ref{equ:fiber-forget}). Furthermore, $\phi^*[\u\Delta'] = 2\cdot [\u\Delta]$.
\end{proposition}
\begin{proof}
Now consider a family of minimal stable log maps $f: C/S \to X$ corresponding to an $S$-point of $F$. Let $\uf: \uC \to \uX$ be the underlying stable map over $\uS$, and $\uf_1: \uC_1 \to \uX$ be the image of $f$ in $\uF'$. Denote by $C^{\sharp} \to S^{\sharp}$ and $C^{\sharp}_1 \to S^{\sharp}_1$ the family of log curves over $\uC \to \uS$ and $\uC_1 \to \uS$ with the canonical log structure. We first notice that there is a canonical commutative diagram of log schemes
\begin{equation}\label{diag:canonical-forget}
\xymatrix{
C^{\sharp} \ar[r] \ar[d] & C^{\sharp}_{1} \ar[d] \\
S^{\sharp} \ar[r] & S^{\sharp}_1
}
\end{equation}
To see this, we may shrink $\uS$, and put $2$ auxiliary markings on the non-contracted component of $\uC$ such that $\uC \to \uS$ is stable, and the components contracted by $\uf$ has no auxiliary markings. 

Indeed, we have a commutative diagram of log stacks
\begin{equation}\label{diag:universal-forget}
\xymatrix{
\cC_{0,5} \ar[r] \ar[d] & \cC_{0,4} \ar[d] \\
\cM_{0,5} \ar[r] & \cM_{0,4}
}
\end{equation}
where $\cC_{0,n} \to \cM_{0,n}$ are the universal family of genus zero stable curves with the canonical log structure. Here the horizontal arrows are obtained by forgetting a marking, and view $\cM_{0,5}$ as the universal curve over $\cM_{0,4}$. Thus, the diagram (\ref{diag:canonical-forget}) is induced by first pulling back (\ref{diag:universal-forget}), then removing the auxiliary markings.

Denote by $F^{\sharp}$ be the log scheme with underlying structure $\uF$, and log structure given by the canonical one of the universal curves. The above argument implies that (\ref{equ:fiber-log-forget}) is given by the composition
\[
F \to F^{\sharp} \to F'
\]
where the first arrow removes the minimal log structure and install the canonical log structure from the curves, and the second one is given by (\ref{diag:canonical-forget}). This is compatible with (\ref{equ:fiber-forget}).

Finally, to compute the pull-back of $\u\Delta'$, we consider the morphism $F^{\sharp} \to F'$. Since the boundary $\u\Delta$ parameterizing curves with two nodes, we have $\oM_{F^{\sharp}}|_{\u\Delta}$ is locally constant with fibers isomorphic to $\NN^2$. In view of (\ref{diag:universal-forget}), fiber-wisely over $\u\Delta$, the morphism $\u\phi^*\oM_{F'} \to \oM_{F^{\sharp}}$ is given by 
\[
\NN \to \NN^{2},  1 \mapsto (1,1).
\]
Combining with Lemma \ref{lem:reducible-conic}, we have the morphism $\u\phi^*\oM_{F'} \to \oM_{F}$ over each geometric point of $\u\Delta$ is given by 
\[
\NN \to \NN, 1 \mapsto 2.
\]

Since both $F$ and $F'$ are log smooth with smooth boundary divisors $\u\Delta$ and $\u\Delta'$ respectively. This implies that $\phi^*[\u\Delta'] = 2\cdot [\u\Delta]$.
\end{proof}

\subsection{Identifying the boundary $\u\Delta$ as a complete intersection}
Fixing a general $(p,q)$, the geometry of the pair $(\uF',\u\Delta')$ has been studied in \cite{Pan}. Let us recall the basic construction. Let $\ell_{pq}$ be the line through $p,q$. And let $$\varphi: \uF'\to \P^{n-2}$$ be the morphism sending each conic to the plane containing it. We may assume that $\P^{n-2}$ is the intersection of the tangent hyperplanes $T_p\uX$ and $T_q\uX$. Consider the following diagram.
\[\xymatrix{
		\u\Delta  \ar[d]^{\u\phi_\Delta}\ar[r]&\uF\ar[d]^{\u\phi}& \\	
		\u\Delta' \ar[rrd]_{\varphi_\Delta} \ar[r]^{} &\uF' \ar[rd]^\varphi&\\
		& & \P^{n-2}}
\]

\begin{proposition}\label{prop:image-Delta}
The following composition:
\begin{equation}\label{equ:boundary-in-proj}
\xymatrix{
		\u\Delta  \ar[r]^{u_\Delta}& \u\Delta'\ar[r]^-{\varphi_\Delta}& \P^{n-2}\subset \P^n
	}
\end{equation}
identifies $\u\Delta$ as a complete intersection in $\PP^n$ of type 
\begin{equation}\label{equ:boundary-type}
(1,1,\cdots,d_1-1,d_1-1,d_1,\cdots,1,1,\cdots,d_c-1,d_c-1,d_c,1).
\end{equation}
\end{proposition}
\begin{proof}
By \cite[Proposition 4.3]{Pan}, the morphism $\varphi:\uF'\to \P^{n-2}$ is a closed embedding. It follows from Proposition \ref{prop:embedding} that the composition (\ref{equ:boundary-in-proj}) is a closed embedding.

The complete intersection type follows from the $\A^1$-line case, and Proposition \ref{prop:lift}. Indeed, we define the functor $$\cR:Sch_\C\to Sets $$ parametrizing families of reducible conics as in the hypothesis of Proposition \ref{prop:lift}. By Proposition \ref{prop:lift}, the functor $\cR$ is isomorphic to the functor associated to $\u\Delta$ under the map $\varphi$. 

Since every $T$-point of $\cR$ is uniquely determined by its node, $\cR$ is isomorphic to the scheme parametrizing the nodes of $\cR$. By Proposition \ref{prop:Fsch}, the locus of the boundary marking of $\A^1$-lines through $p$ (or $q$) is a complete intersection of type $(1,\cdots,d_1,\cdots,1,\cdots,d_c,1).$ After combining these polynomials, there is a redundancy $(d_1,\cdots,d_c,1)$ saying that $r$ lies on $\uD$. Now the proposition follows. 
\end{proof}

\subsection{Degree of $\u\Delta'$}
Let $\uM\subset Hilb^{2t+1}(\P^n)$ be the moduli space of conics in $\P^n$ passing through $p,q$. Let $\P^{n-2}$ be any projective subspace in $\P^n$ which does not intersect $\ell_{pq}$. We have a canonical morphism $$h:\uM\to \P^{n-2}$$ 
which maps a conic to the plane it expands. It follows that $\uM$ is a $\P^3$-bundle over $\P^{n-2}$, as the conics need to pass through two general points $p,q$. Let $\kappa$ be the relative $\O(1)$-bundle of $h$. 

\begin{lemma}\label{lem:deg2}Let $\uM_0$ be the open subset of $\uM$ parameterizing conics which do not contain the line $\ell_{pq}$.  Let $\u\Delta(M_0)\subset \uM_0$ be the closed subset of reducible conics. Let $\u\Delta(M)$ be the closure of $\u\Delta(M_0)$ in $\uM$.
	Then we have the following: 
\begin{enumerate}
		\item $\uM_0$ is an $\A^3$-bundle over $\P^{n-2}$.
		\item $\u\Delta(M)$ is a smooth quadric surface bundle over $\P^{n-2}$.
		\item $\u\Delta(M)$ is linearly equivalent to $2\kappa+2h^*(\O_{\P^{n-2}}(1))$. In particular, $\u\Delta(M_0)$ as a divisor in $\uM_0$ is linearly equivalent to $2h^*(\O_{\P^{n-2}}(1))$.
	\end{enumerate} 
\end{lemma}

\proof The first two statements follows from computation of plane conics. Indeed, fix a plane in $\P^n$ containing $p,q$. Let $p=[0:1:0]$ and $q=[0:0:1]$. The plane conics through $p$ and $q$ is of the form:
\[
a_1 x^2 + a_2 xy + a_3 xz + a_4 yz = 0
\]
It is reducible if and only if either $a_4 =0$ or $a_1 a_4+a_2 a_3=0$. The first case corresponds to the locus parameterizing conics containing $\ell_{pq}$, while the second case does not. This proves statements (1) and (2).

From the above calculation, the divisor $\Delta(M)$ is linearly equivalent to $2\kappa+c \cdot h^*(\O_{\P^{n-2}}(1))$ for some coefficient $c$. To determine $c$, we constrct a testing curve, and check its intersection numbers with $\Delta(M)$ as follows. We pick a general line $L$ on $\P^{n-2}$, and a general smooth quadric hypersurface $Q$ in $\P^{n}$ containing $p,q$. For any point $t$ on $L$, the plane $H_{pqt} \cong \P^2$ spanned by the three points $p,q,$ and $t$ intersects $Q$ at a conic $C_t$ through $p,q$. Furthermore, $\{C_t\}$ is a pencil of conics lying on the quadric surface which is the intersection of $Q$ and the span of $\ell_{pq}$ and $L$. Therefore, there are two reducible conics in this pencil. Furthermore, this intersection is transversal because the testing pencil is free in $M$, i.e., it can pass through a general point of $M$.\qed

\begin{proposition}\label{prop:deg2}
	The smooth divisor $\u\Delta'\subset \uF'$ is cut out by a homogeneous polynomial of degree two.
\end{proposition}

\proof This is proved in \cite[Theorem 6.17]{Pan}. Here we present a simple proof. Consider the commutative diagram
	$$\xymatrix{
		\u\Delta' \ar[d] \ar[r]^{} &\uF' \ar[d] \ar[rd]^\varphi&\\
	\u\Delta(M_0)\ar[r]	&\uM_0\ar[r]^h & \P^{n-2}}$$
where the left square is Cartesian. Since $\u\Delta(M_0)=h^*(\O_{\P^{n-2}}(2))$, so is $\u\Delta'$.\qed

\subsection{Proof of Theorem \ref{thm:A1-conic-moduli}}
\begin{proposition}\label{prop:deg1}
	The smooth divisor $\u\Delta\subset \uF$ is cut out by a linear form in $\P^{n-2}$.
\end{proposition}
\proof This follows from Proposition \ref{prop:boundary-pullback} and \ref{prop:deg2}.\qed

\proof[Proof of Theorem \ref{thm:A1-conic-moduli}] By \cite[Proposition 4.3]{Pan}, Proposition \ref{prop:image-Delta}, and \ref{prop:deg1}, we have:
\begin{itemize}
	\item $\uF\subset \P^{n-2} \subset \P^n$ is a smooth projective variety.
	\item $\u\Delta=W\subset\P^{n-2} \subset \P^n$ is a complete intersection of type 	$$(1,1,\cdots,d_1-1,d_1-1,d_1,\cdots,1,1,\cdots,d_c-1,d_c-1,d_c,1).$$
	\item $\u\Delta$ is cut out by a linear form on $\uF$.
\end{itemize}
Now the theorem follows from \cite[Proposition 7.1]{Pan}.\qed

\section{Proof of Theorem \ref{thm:main} and its corollaries}\label{sec:proof-conclude}

\proof[Proof of Theorem \ref{thm:main}] Since affine spaces satisfy strong approximation \cite[Theorem 6.13]{Rosen2002}, it remains to prove under Assumption \ref{not:deg1pair}. By Theorem \ref{thm:A1-conic-moduli} and the hypothesis on the degree, the general fiber of the evaluation morphism 
$$
\u{ev}:\u\cA_{2}(X,2\alpha)\to \uX\times \uX
$$ is a smooth Fano complete intersection. Therefore, the general fiber is rationally connected by \cite{Campana, KMM}. On the other hand, we know that weak approximation holds for $X$ \cite{dJS,Hassett-WA}. Now the theorem follows from Theorem \ref{thm:SA}.\qed

\proof[Proof of Corollary \ref{cor:main}] For any smooth complete intersection pair $X=(\uX,\uD)$ in $\P^n$ with coordinate $[x_0:\cdots:x_n]$, we assume that 
$$\uX=\{F_1=F_2=\cdots=F_c=0\}$$ 
where $\deg F_i=d_i$ and the boundary $\uD=\{G=0\}$ where $\deg G=k$. The universal cover of $X\setminus D$ can be constructed in $\P^{n+1}$ with the coordinate $[x_0:\cdots:x_n:y]$ by taking the complete intersection
\begin{equation}\label{equ:universal-cover}
\uY=\{F_1=\cdots=F_c=y^k-G=0  \}
\end{equation} 
with the boundary divisor
\begin{equation}\label{equ:universal-cover-boundary}
\uE=\{y=0 \}.
\end{equation} 
We check that $(\uY,\uE)$ is smooth complete intersection pair in $\P^{n+1}$ of type $(d_1,\cdots,d_c,k;1)$. Furthermore the natural projection to $\P^n$ defined by $y=0$ gives a cyclic branched cover of degree $k$ over $X$. This yields the universal cover $Y \setminus E \to X\setminus D$. Now the corollary follows from Theorem \ref{thm:main}.\qed

\proof[Proof of Corollary \ref{cor:geom}] 
Let $Y = (\uY,\uE)$ be the universal cover constructed in (\ref{equ:universal-cover}) and (\ref{equ:universal-cover-boundary}). For each given point $x_i\in X^\circ$, $i=1,\cdots, m$, we may choose a lift of $y_i\in Y^\circ$. Since strong approximation holds for the constant family $\pi:Y\times \P^1\to \P^1$ away from $S=\{\infty \}$ by Corollary \ref{cor:main}, there exists an integral section curve $C$ passing through $(y_1,t_1),\cdots,(y_m,t_m)$, where $t_i$'s are distinct points on $\P^1-\{\infty \}$. The projection $p_1(C)$ gives an $\A^1$-curve on $Y$ passing through $y_1,\cdots,y_m$. Composing it with the map from $Y$ to $X$ gives the desired $\A^1$-curve.\qed


\bibliographystyle{amsalpha}             
\bibliography{myref}

\end{document}